\documentclass[]{amsart}
\usepackage{latexsym,fancyhdr,amssymb,color,amsmath,amsthm,graphicx,listings,comment}
\usepackage[section]{placeins}
\newtheorem{thm}{Theorem} \newtheorem{lemma}{Lemma} \newtheorem{coro}{Corollary}
\let\paragraph\subsection

\title{Complexes, Graphs, Homotopy, Products and Shannon capacity}
\author{Oliver Knill}
\date{December 13, 2020}
\address{Department of Mathematics \\ Harvard University \\ Cambridge, MA, 02138 }


\begin{document}
\maketitle

\begin{abstract}
A finite abstract simplicial complex $G$ defines the
Barycentric refinement graph $\phi(G) = (G,\{ (a,b), a \subset b \; {\rm or} \; b \subset a \})$ and the connection graph 
$\psi(G) = (G,\{ (a,b), a \cap b \neq \emptyset \})$. We note here that both functors $\phi$ and $\psi$
from complexes to graphs are invertible on the image (Theorem 1) and that $G,\phi(G),\psi(G)$ all have the same 
automorphism group and that the Cartesian product of $G$ corresponding to the Stanley-Reisner product of $\phi(G)$ and
the strong Shannon product of $\psi(G)$, have the product automorphism groups. 
Second, we see that if $G$ is a Barycentric refinement, then $\phi(G)$ and $\psi(G)$ are graph homotopic 
(Theorem 2). Third, if $\gamma$ is the geometric realization functor, 
assigning to a complex or to a graph the geometric realization of its clique complex, then 
$\gamma(G)$ and $\gamma(\phi(G))$ and $\gamma(\psi(G))$ are all classically homotopic 
for a Barycentric refined simplicial complex $G$ (Theorem 3). The Barycentric assumption is necessary in Theorem 2 and 3.
There is compatibility with Cartesian products of complexes which manifests in the strong graph
product of connection graphs: if two graphs $A,A'$ are homotopic and $B,B'$ are homotopic, then 
$A \cdot B$ is homotopic to $A' \cdot B'$ (Theorem 4) leading to a commutative ring of homotopy classes of graphs.
Finally, we note (Theorem 5) that for all simplicial complexes $G$ as well as product $G=G_1 \times G_2 \cdots \times G_k$,
the Shannon capacity $\Theta(\psi(G))$ of $\psi(G)$ is equal to the number $f_0$ of zero-dimensional 
sets in $G$. An explicit Lowasz umbrella in $\mathbb{R}^{f_0}$ leads to the Lowasz number $\theta(G) \leq f_0$ and
so $\Theta(\psi(G))=\theta(\psi(G))=f_0$ making $\Theta$ compatible with disjoint union addition and strong multiplication.
\end{abstract}

\section{Theorem 1}

\paragraph{}
A finite set $G$ of non-empty sets that is closed under the operation of 
taking finite non-empty subsets is called a {\bf finite abstract simplicial complex}.
It defines a finite simple {\bf connection graph} $\psi(G)$ in which the vertices are the elements 
in $G$ and where two sets are connected if they intersect. In the {\bf Barycentric refinement graph}
$\phi(G)$, two vertices are connected if and only if one set is contained in the other. 
It is a subgraph of the connection graph $\psi(G)$.

\begin{thm}
$G$ can be recovered both from $\psi(G)$ or $\phi(G)$. 
\end{thm}

\paragraph{}
The proof will be obvious, once the idea is seen. The reconstructions from $\phi(G)$ or $\psi(G)$ 
are identical. The reconstruction can be done fast, meaning that the cost is polynomial
in $n=|G|$. In particular, no computationally hard {\bf clique finding} is necessary in $\psi(G)$.
When looking at a graph like $\psi(G)$, we of course only assume to know the graph structure and not 
what set each node represents. As no explicit reconstruction of $G$ from $\psi(G)$ appears 
to have been written down before, this is done here. It will show that it is considerably simpler
than the construction of a set of sets from a general graph that is enabled by the 
{\bf Szpilrajn-Marczewski theorem}: any finite simple graph $A$ can be realized as a connection graph 
of a finite set $G$ of non-empty sets \cite{Szipilrajn-Marczewski,ErdoesGoodmanPosa}. Connection graphs are 
special: their adjacency matrix $A$ has the property that $L=1+A$ has determinant $1$ or $-1$ and 
the number of positive eigenvalues of $L$ is the number of even-dimensional sets in $G$.  

\paragraph{}
The {\bf automorphism group} ${\rm Aut}(G)$ is the set of permutations $T$ of $G$ which preserve the order structure:
$x \subset y$ if and only if $T(x) \subset T(y)$. The automorphism group of a graph is the automorphism
group of its Whitney simplicial complex. This is the same than the automorphism of its 1-dimensional skeleton
complex $G=V \cup E$ because if edges are mapped to edges then also complete graphs are mapped into complete graphs.
An automorphism $T$ of a graph is nothing else than a map from the graph to itself which is an isomorphism. 
A consequence of the proof is that $G,\phi(G),\psi(G)$ all have the same automorphism groups. 
Groups like ${\rm Aug}(G)$ is in the {\bf Klein Erlanger picture} an important object of a geometry as ${\rm Aut}(G)$ 
is a {\bf symmetry group}. {\bf Frucht's theorem} shows that any finite group can occur for a graph. By building the 
Whitney complex of this graph, we see that an8y finite group can occur as an automorphism group of a simplicial
complex and so also of a connection graph $\psi(G)$ or a Barycentric graph $\phi(G)$. 

\begin{coro}
For any simplicial complex $G$, all $\psi(G)$ and $\phi(G)$ have the same automorphism group. 
\end{coro}
\begin{proof}
If $T$ is an automorphism of $G$, then it produces an automorphism both on 
$\phi(G)$ and $\psi(G)$. On the other hand, if we have an automorphism of 
a graph then the reconstruction allows to transport this automorphism to $G$:
$T$ commutes the vertices which produces a permutation of the vertex set $G_0$
of $G$. This define uniquely the permutation on $G$. 
\end{proof}

\paragraph{}
In the case when $G$ is a Barycentric refinement, then also the Lefschetz number $L(T,G)$ 
(the super trace of the induced map on the cohomology groups) is the same than the Lefschetz number 
$L(T,\phi(G))$ or $L(T,\psi(G))$ and if it is not zero, like for a contractible complex, there is at
least one fixed point, because of the Lefschetz fixed point theorem $L(T,G) = \sum_{x, T(x)=x} i_T(x)$
where $i_T(x) = (-1)^{{\rm dim}(x)} {\rm sign}(T:x \to x)$ is the index of the fixed point $x$
which in the case of a graph is a fixed clique. See \cite{brouwergraph}. We will also see

\begin{coro}
The Cartesian product $G_1 \times G_2$, the Stanley-Reisner product $\phi(G_1 \times G_2) = \phi(G_1)  \cdot  \phi(G_2)$ 
as well as the strong product $\psi(G_1) \cdot \psi(G_2)$ have the same automorphism group which is the 
product group of the automorphism groups of $G_1$ and $G_2$. 
\end{coro}

\begin{figure}[!htpb]
\scalebox{1.0}{\includegraphics{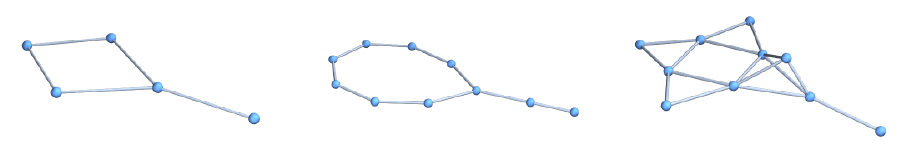}}
\label{Figure 1}
\caption{
The complex $G=\{ (1,2),(2,3),(3,4),(4,1),(4,5),(1),(2),(3),(4),(5) \}$, its
Barycentric refinement graph $\phi(G)$ and the connection graph $\psi(G)$.
The first picture visualizes $G$ also as a graph as $G$ is the Whitney complex
of that graph. We know  by Theorem 5 that
$\Theta(\psi(G))=\Theta(\phi(G))=5$. 
}
\end{figure}

\paragraph{}
Theorem 1 shows that one does not lose information by looking at connection graphs
of simplicial complexes. Finite abstract simplicial complexes have only one axiom and 
still allow to explore interesting topology, like for example the topology of compact 
differentiable manifolds. The simple set-up for finite abstract simplicial 
complexes is even simpler than Euclid's axiom system about points and lines.
When thinking about sets of sets, it is helpful to look at incidence and intersection 
graphs because sets by themselves evoke little geometric intuition. The graphs
$\phi(G)$ and $\psi(G)$ help for that and build a link to the actual topology. 
The graph $\phi(G)$ for example is all we need to get all the cohomology groups which 
are kernels of block matrices of the Hodge Laplacian $(d+d^*)^2$ invoking the incidence
structure. The intersection structure is spectrally natural: because the
product for connection graphs produces spectral data which multiply. 

\paragraph{}
We will discuss in a moment the topology and homotopy of $\phi(G)$ and $\psi(G)$. For now, note that the
graph $\psi(G)$ is topologically quite different from $G$ already for $1$-dimensional simplicial complexes.
While the functor $\phi$ from simplicial complexes to graphs does honor the maximal dimension, 
(the dimension of the maximal complete subgraph), 
the functor $\psi$ does not, as simple examples show. For example, if $x$ is a set
in $G$ which intersects with $n$ other sets, then there are $n+1$ sets which all 
intersect with each other, so that the graph $\phi(G)$ has a clique of dimension $n$.
For the $1$-dimensional {\bf star complex} $G=\{ (1,2),(1,3),\dots,(1,n) \}$ for example, 
the connection graph $\psi(G)$ already has maximal dimension $n$. The next theorem will
show however that $\psi(G)$ is graph homotopic to $G$, implying that the geometric realizations
of $\psi(G)$ and $G$ are classically homotopic.

\paragraph{}
This note justifies partly some statements related to simplicial complexes and graphs
\cite{KnillBaltimore,KnillILAS,AmazingWorld}. It is a small brick in a larger building hopefully
to emerge at some point. We also hope to pitch the simplest homotopy set-up 
in mathematics and illustrate with some lemmas how to work effectively with it (an Appendix gives an other example).
We will see in particular that important deformation procedures for graphs are homotopies: examples are
{\bf edge refinements}, which serve as local Barycentric refinements, as well as global Barycentric refinements.
These two deformations even preserve the manifold structure of graphs in the sense that unit spheres remain
spheres and keep the dimension during the deformation. Homotopies 
(like $K_n \to K_{n+1}$) of course do not preserve dimension in general.  
We add in an appendix more about discrete spheres. 

\paragraph{}
For us, it is important that we can make sense of {\bf Cartesian products} of simplicial complexes 
without having to dive into other discrete combinatorial notions like {\bf simplicial sets}
or {\bf discrete CW complexes} (this has been used 
in \cite{KnillEnergy2020}) which are both combinatorial 
useful and categorically natural but also require more 
mathematical sophistication. We don't want to use 
the geometric realization functor to prove things in combinatorics. 

\section{Theorem 2}

\paragraph{}
If $x$ is a vertex of a finite simple graph $B=(V,E)$, let $S(x)$ denote the {\bf unit sphere} of $x$.
It is the graph induced from the set of all the vertices connected to $x$. The class of {\bf contractible graphs} is
defined recursively: the $1$-point graph $K_1=\{ \{1\},\{\} \}$ is contractible.
A graph $(V,E)$ is called {\bf contractible}, 
if there exists $v \in V$ such that the unit sphere $S(v)$, (the graph induced by the neighbors of $v$),
as well as the graph $B-v$, (the graph induced by $V \setminus \{v\}$), are both contractible. 
Since both $S(v)$ and $B-v$ have less vertices, the inductive definition works and allows to check 
contractibility in polynomial time with respect to the number of vertices in the graph. 

\paragraph{}
A {\bf homotopy step} is the process of removing a vertex $v$ for which $S(v)$ is contractible
or the reverse procedure of choosing a contractible subgraph $A$ of $B$ and connecting every 
vertex in $A$ to a new vertex $v$. Two graphs $A,B$ are called {\bf homotopic}, if one can chose a finite set of
homotopy steps to get from $A$ to $B$. This homotopy emerged from \cite{Ivashchenko1993,I94a,CYY}, is based on 
\cite{WhiteheadI} and already appears in \cite{Graham1979} according to \cite{CYY}. It is simple and fully
equivalent to the continuum homotopy and defines also a homotopy for simplicial complexes: two complexes $G,H$ 
are declared to be homotopic if $\phi(G)$ and $\phi(H)$ are homotopic graphs. 

\paragraph{}
An different homotopy was suggested in \cite{BBLL}; it is based on product graphs, closer to the continuum but
harder to implement. We have used the above mentioned homotopy since \cite{josellisknill,KnillTopology}. See
\cite{KnillBaltimore,KnillILAS,AmazingWorld} for presentations when discussing
coloring problems \cite{knillgraphcoloring,knillgraphcoloring2,knillgraphcoloring3}. In the appendix, we
illustrate how homotopy defines spheres and allows to prove properties of spheres like the Euler Gem formula. 
The Appendix was a talk and gives all definitions of the two classes ``spheres" and ``contractible graphs".
The Ljusternik-Schnirelman point of view is that contractible spaces have category 1, and spheres have
category 2 with the exception of the $(-1)$-sphere, the empty graph which has L-S category 0. 

\paragraph{}
Contractible graphs by definition are homotopic
to $1$ (the one-point graph $K_1$ with only one vertex and no edges) 
but the graphs like the {\bf dunce hat} are homotopic to $1$ but not contractible. It is necessary
first to thicken up a graph in general before it can be contracted. The class of graphs 
which are homotopic to $1$ form a much larger class of graphs than the set of contractible graphs
and are inaccessible in the sense that it is a computational hard problem to decide whether
a graph is homotopic to $1$ or not. Contractibility on the other hand is decidable:
as we only need to go through all vertices and check whether their unit spheres are contractible
and because unit spheres have less vertices. 

\begin{thm}
For $G$ Barycentric, $\phi(G)$ and $\psi(G)$ are homotopic. 
\end{thm}

\begin{figure}[!htpb]
\scalebox{1.0}{\includegraphics{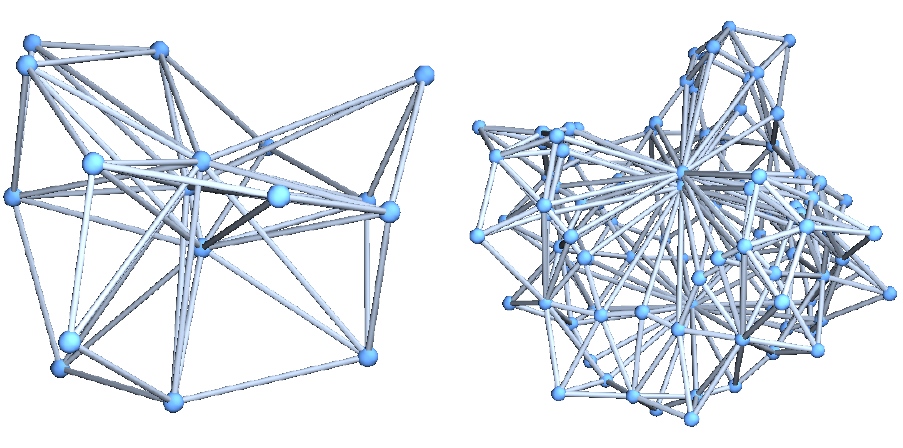}}
\label{Figure 8}
\caption{
The dunce hat is a graph with 17 vertices and 52 edges and 36 triangles,
which is homotopic to $1$ but not contractible. 
All its Barycentric refinement are homotopic to $1$ but
not contractible. 
}
\end{figure}

\paragraph{}
The 1-dimensional complex 
$$  G=\{(1,2),(2,3),(3,1),(1),(2),(3)\} $$
is the boundary complex of a triangle $K_3$ and topologically a $1$-sphere (circle).
The graphs $\phi(G)=C_6$ and $\psi(G)$ are not
homotopic because $\chi(\phi(G))=0$ and $\chi(\psi(G))=1$ and Euler characteristic
$\chi$ is a homotopy invariant. 
For {\bf octahedron complex} which has six $0$-dimensional vertices, 
the graph $\phi(G)$ has the topology of a $2$-sphere while $\psi(G)$ is homotopic to a 
$3$-sphere. The proof of Theorem~2 is not difficult, once one sees how the homotopy steps are done. 
We will write down the concrete graph homotopy.

\begin{figure}[!htpb]
\scalebox{1.0}{\includegraphics{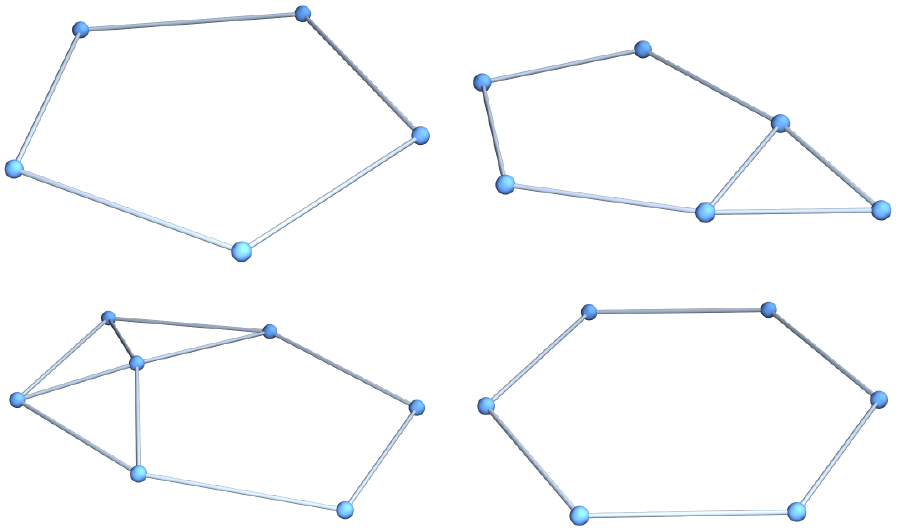}}
\label{Figure 4}
\caption{
The homotopy deformation from the circular graph $C_5$ to $C_6$ needs 
three homotopy steps. It is not possible to contract $C_6$
to $C_5$ as the unit spheres of all vertices are $0$-spheres and 
not contractible. The graph first needs to be thickened up at first and
temporarily becomes two-dimensional. 
}
\end{figure}

\paragraph{}
Homotopy preserves {\bf Euler characteristic} $\chi(A) = \sum_{x \in A} (-1)^{{\rm dim}(x)}$, 
summing over all complete subgraphs $x$ of $A$. That homotopy is an invariant follows directly from
$$ \chi(B +_A v) = \chi(B) + (1-\chi(A) $$ 
if $B +_A v$ is the graph in which a new vertex is attached to a subgraph $A$ of $B$. 
If $A$ is contractible, then $\chi(A)=1$ and $\chi(B)=\chi(B+v)$. 
This formula is a direct consequence of the {\bf valuation property}
$\chi(A \cup B) = \chi(A) + \chi(B) - \chi(A \cap B)$ for any subgraphs $A,B$ of
a larger graph. When applying it to a build up of the 
complex, it implies the {\bf Poincar\'e-Hopf formula} \cite{poincarehopf,MorePoincareHopf}
$$ \chi(A) = \sum_{v \in V(A)} i_f(x) $$
for a locally injective function $f$ on $V(\Gamma)$, where $i_f(x) = 1-\chi(S_f(x))$ is the 
{\bf Poincar\'e-Hopf index} of $f$ at $x$ and 
$S_f(x)$ is the subgraph generated by all $y \in S(x)$, where $f(y)<f(x)$. 

\begin{coro}
If $G$ is Barycentric then $\chi(G) = \chi(\psi(G)) = \chi(\phi(G))$. 
\end{coro}

\paragraph{}
A consequence is that all cohomology 
groups and their dimensions, the {\bf Betti numbers} are invariant under homotopy deformations. 
This can be verified also within the discrete setup: extend any cocycle from 
$G$ to $G+_A x$ and also extend every coboundary from $G$ to $G+_A x$. Also discrete Hodge theory 
works: deform the kernel of the blocks $H_k(G)$ of $H(G)$ to the kernels of $H_k(G+_A x)$: 
just let the heat flow $e^{-t H}$ act on a function $f$ on $G$ that had been harmonic and 
initially was extended to $G+_A x$ by assigning
$0$ to every $k$-simplex in $G+_A x$ not in $G$. The heat flow will deform the function to a 
harmonic function on $G+_A x$. 

\begin{figure}[!htpb]
\scalebox{0.4}{\includegraphics{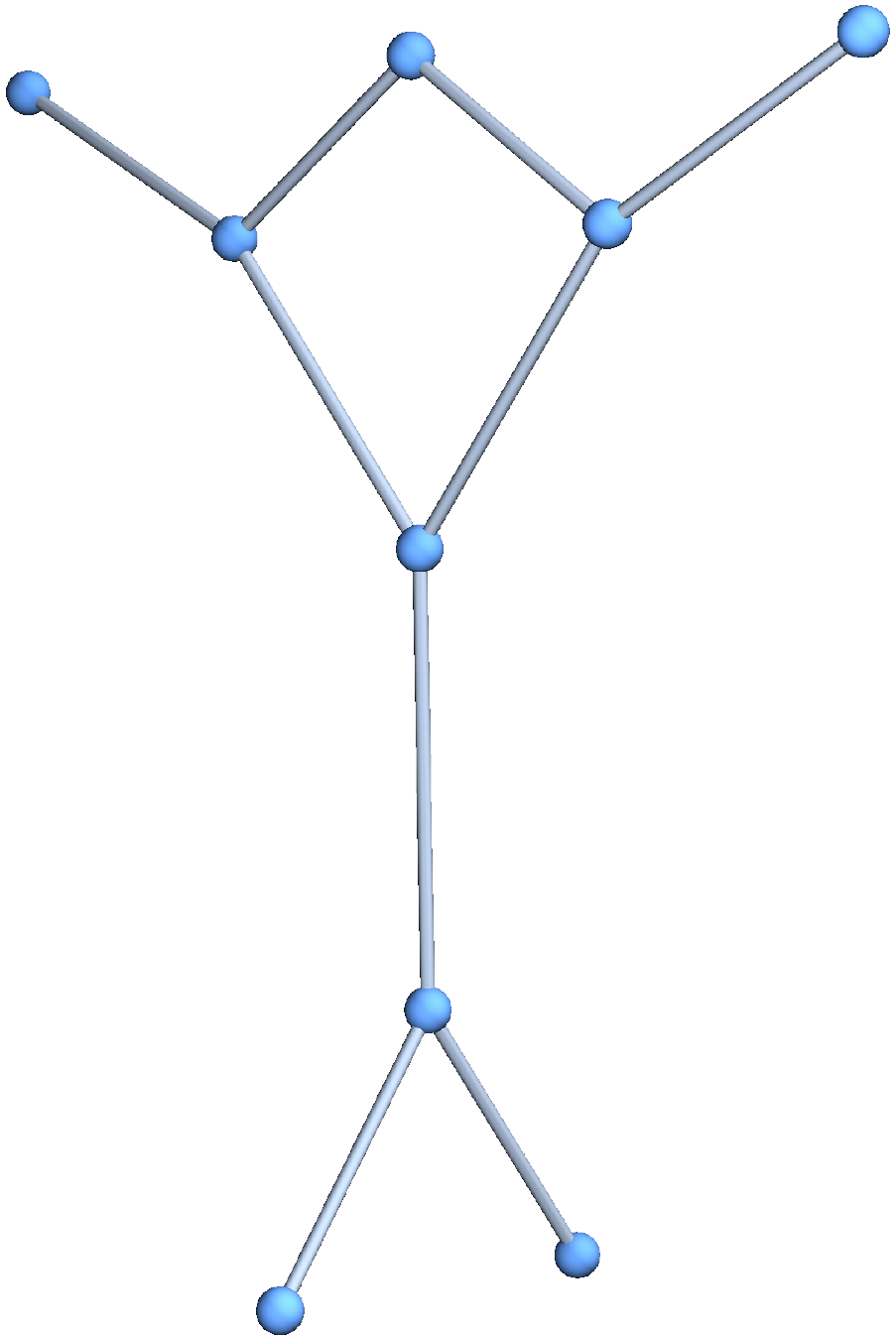}}
\scalebox{0.4}{\includegraphics{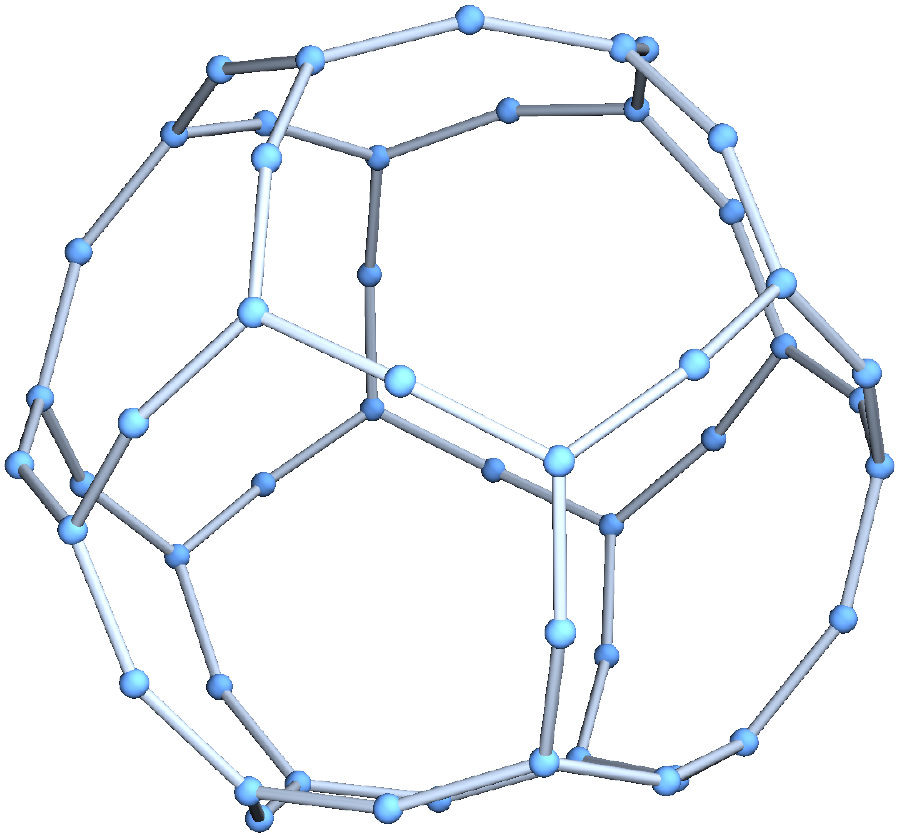}}
\label{Figure 9}
\caption{
For $1$-dimensional connected simplicial complexes (curves), the
Betti number $b_1$ is the genus, the number 
of holes completely determines the homotopy type. Alternatively $1-b_1=\chi(G)$, the Euler characteristic
characterizes connected ``curves". All trees are contractible, the graph in the
first picture has $b_1=1$ and is homotopic to a circle. The second $1$-dimensional complex has
the Betti numbers $(1,11)$ and Euler characteristic $1-11=-10$. We could fill the $12$ holes 
and get a $2$-sphere of Euler characteristic $2$. 
}
\end{figure}

\section{Theorem 3}

\paragraph{}
In this section, we leave combinatorics in order to illustrate the connection with
topology. Mathematicians like \cite{AlexandroffHopf} thought in terms of discrete graphs
(i.e. Eckpunktger\"ust) and this is still visible in modern algebraic topology \cite{Hatcher} and especially
texts which show in drawings how topologists think \cite{Fomenko}. 
The geometric Whitney realization of a graph $A$ is a union $\gamma(A)$ of simplices in some Euclidean
space $\mathbb{R}^n$ such that every complete subgraph in $A$ is mapped into a simplex. The simplest way to 
do that is to see the clique complex of $A$ as a subcomplex of the maximal simplex in the 
complete graph with vertices in $A$.

\begin{thm}[Theorem 3]
If two graphs $A$ and $B$ are homotopic, their geometric Whitney realizations
$\gamma(A),\gamma(B)$ are homotopic topological spaces.
\end{thm}

\paragraph{}
Because contractible and collapsible are are used differently in the literature and are easily confused, we
use contractible and collapsible as a synonym and use {\bf ``homotopic to $1$"} if a graph is
homotopic to $1$. The equivalence relation ``homotopic to $1$" is much harder to check than
being contractible. While we can by brute
force decide in a finite number of steps whether a graph is contractible (just go through all the vertices and see whether
one can remove it by checking its unit sphere to be contractible), the difficulty with the wider homotopy
is that we possibly have to expand the graph first considerably before we can contract.

\paragraph{}
It is already for $2$-dimensional complexes known to be undecidable in the sense that there is no Turing
machine which takes as an input a finite simple graph and as an output the decision whether it is
homotopic to $1$ or not. This work has started with Max Dehn and relates to other problems like
the triviality of the fundamental group which can be related to word problems in groups. 
Already for $2$-dimensional simplicial complexes, the problem to decide simply connectedness is 
algorithmically unsolvable. Lets abbreviate ``Barycentric $G$" for ``Barycentric refined $G$``.

\begin{coro}
If $G$ is Barycentric, then $\gamma(G),\gamma(\phi(G))$ and $\gamma(\psi(G))$ are all 
homotopic topological spaces.
\end{coro}

\section{Theorem 4}

\paragraph{}
The {\bf strong product} of two graphs $(V,E),(W,F)$ is the graph $(V \times W, Q)$, where
$Q=\{ ( (a,b), (c,d) ),   b=d \; {\rm or} \; (b,d) \in F$ and $a=c \; {\rm or} \; (a,c) \in E \}$. 
It is an associative operation which together with disjoint union $+$, (the monoid is group completed to an additive 
group), produces the {\bf strong ring} of graphs. It is a commutative ring with $1$-element  
$1=K_1$ and where $0$ is the empty graph. While the Cartesian product $G \times H$ of simplicial complexes
is not a simplicial complex, one has a product $\phi(G) \times \phi(H)$ on the Barycentric graph level
$(G \times H, \{ ( (a,b), (c,d) ),   a \subset c,$ and $b \subset d)$. One can compare this with 
$\psi(G)  \cdot  \psi(H) = (G \times H, ( (a,b),(c,d) ),   a \cap c \neq \emptyset, b \subset d \neq \emptyset \}$. 
The later graph $\psi(G)  \cdot  \psi(H)$ is spectrally nice in that the connection Laplacians 
$L(G \times H)$ is the tensor product of $L(G)$ and $L(H)$. (See \cite{KnillEnergy2020}).

\begin{thm}[Theorem 4]
If $A,A'$ are homotopic and $B,B'$ are homotopic, then $A \cdot B$ is homotopic to $A' \cdot B'$. 
\end{thm}

\paragraph{}
It follows that the strong ring of graphs defines also a {\bf ring of homotopy classes} of 
signed graphs. The homotopy of a signed graph $A-B$ allows to deform $A$ and $B$. If $A,B$ are
homotopic then $A-B$ is $0$. The additive primes in the ring of the classes of connected graphs, 
the multiplicative primes are the homotopy classes graphs which come from finite abstract
simplicial complexes. The $1$-element in the ring is the class of all graphs which are 
homotopic to $1$. 

\begin{figure}[!htpb]
\scalebox{1.0}{\includegraphics{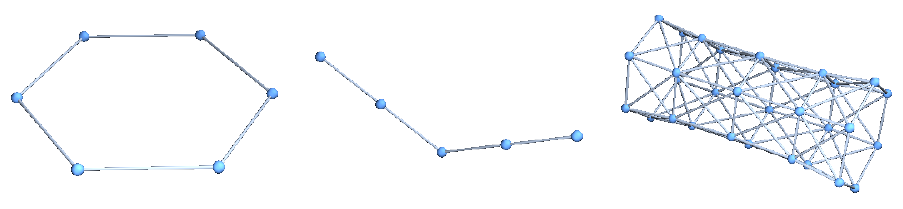}}
\label{Figure 11}
\caption{
The strong product a circular graph $C_6$ and a linear graph 
$L_5$ is a discrete cylinder.
}
\end{figure}

\begin{figure}[!htpb]
\scalebox{1.0}{\includegraphics{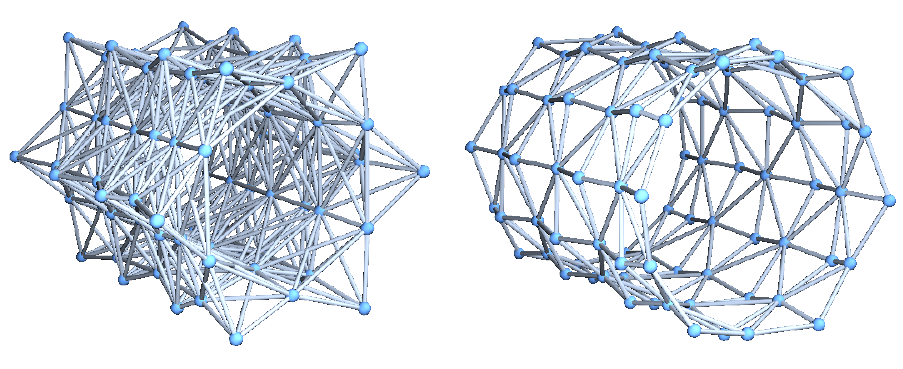}}
\label{Figure 12}
\caption{
The strong product $\psi(G \times H)$ (left) of a circular complex $G$ and linear complex $H$
compared with the Barycentric product $\phi(G \times H)$ (right).
Both have the same number of vertices and are homotopic
(Theorem 2). Only $\phi(G \times H)$ is a discrete manifold with boundary but
$\psi(G \times H)$ is the strong product of $\psi(G)$ and $\psi(H)$. 
}
\end{figure}

\section{Theorem 5}

\paragraph{}
We can assign to a product of simplicial complexes 
$G \times H$ (the Cartesian product as sets is not a simplicial complex) the graph 
$\psi(G) \cdot \psi(H)$, which is the {\bf strong product} 
of the two connection graphs. This product
appeared already in work of Shannon \cite{Shannon1956} or Sabidussi \cite{Sabidussi}.
Shannon used it when defining Shannon capacity of a graph $A$ as the asymptotic growth of the 
{\bf independence number} of the $n$'th power $A^n$ of a complex. 
This number $\Theta(A)$ is in general difficult to compute:
Shannon himself computed it for all graphs with less than $5$ points and estimated
$\sqrt{5} \leq \Theta(C_5) \leq 5/2$. Only 23 years later, $\Theta(C_5)=\sqrt{5}$ was proven
\cite{Lovasz1979} and $\Theta(C_7)$ is still unknown \cite{Matousek}. Lovasz introduced the
Lowasz number $\theta(G)$ which is the $\sec(\theta)^2$ of the maximal angle of the opening 
angle of the {\bf Lovasz umbrella}, a geometric cone shaped object obtained from an orthonormal representation
of the graph, and satisfies $i(G) \leq \Theta(G) \leq \theta(G) \leq c(\overline{G})$, where $c(G)$ is the
chromatic number also dubbed {\bf sandwich theorem} \cite{KnuthSandwich} showing that $\Theta$ has lower
and upper bounds which are NP hard but also has an upper bound $\theta(G)$ which can be computed in polynomial
time. 

\begin{figure}[!htpb]
\scalebox{1.0}{\includegraphics{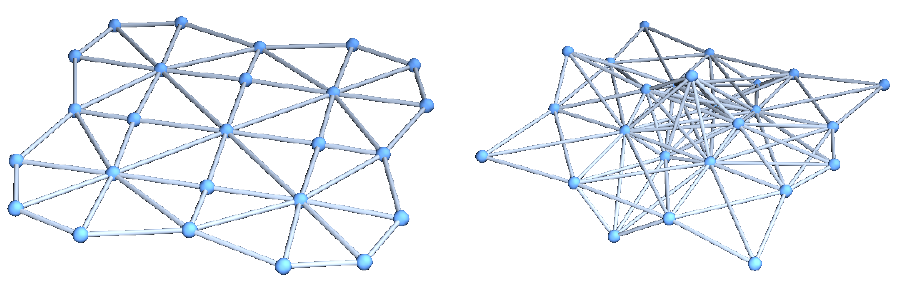}}
\label{Figure 13}
\caption{
We see $\phi(G \times H)$ and $\psi(G \times H)$ for $G=H=\{ (1,2),(2,3),(1),(2),(3) \}$. 
The Barycentric product is the Stanley-Reisner product multiplying $G=x+y+z+xy+yz$,
$H=u+v+w+uv+vw$ and connecting monomials which divide each other. 
The right graph $\psi(G \times H)$ is homotopic to 
$\phi(G \times H)$ and is equal to $\psi(G) \cdot \psi(H)$. 
Not only cohomology \cite{KnillKuenneth} but also the spectral compatibility is
satisfied as the connection Laplacian $L(G \times H) = L(G) \otimes L(H)$ is the tensor product.
We have $\Theta( \phi(G \times H)) = \Theta(\psi(G \times H)) = f_0(G) f_0(H)=3 \cdot 3=9$. 
}
\end{figure}

\paragraph{}
The independence number $i(A)$ of a graph is $A$ the clique number 
of the graph complement $\overline{A}$.
Because clique numbers are hard to computer, also independence numbers are 
difficult to get. We expect especially the {\bf Shannon capacity}
$$  \log(\Theta(A)) = \lim_{n \to \infty} \frac{1}{n} \log(i(A^n))   \;  $$
to be difficult to compute in general. Shannon used the log \cite{Shannon1956},
while in modern treatments like \cite{Matousek} one looks at $\Theta(A)$, which we
use too. Here is a bit of a surprise. For connection graphs as 
well as Barycentric graphs, we know the capacity explicitly. 

\paragraph{}
Let $f_0(G)$ denote the number of $0$-dimensional sets in $G$. 

\begin{thm}[Theorem 5]
We have $\Theta(\psi(G))=f_0(G)$.
\end{thm}

\paragraph{}
So, while for some $G=C_n$, we can not compute $\Theta(G)$, we can
do it for $\psi(G)$. We know the capacity for complete graphs $A=K_n$, for $A=C_{2n}$ the 
linear graph $L_{n}$, as well as all connection graphs 
$A=\psi(G_1 \times \cdots \times G_n))$ of products of 
simplicial complexes have the property $\Theta(A) = i(A)$. This begs for the question: 
which connected graphs have the property that their Shannon capacity is 
equal to the independence number? 
The above list is far from complete: for example take $K_5$ and remove edges to get a connected
graph making sure that $i(A)$ remains $2$. Then $\Theta(A)=2$ also. 
These are graphs for which communication capacity can not be increased 
by taking products.

\paragraph{}
While $\Theta(A+B) \geq \Theta(A) + \Theta(B)$ in general, it is a bit surprising that 
$\Theta(A+B)$ can be strictly larger than $\Theta(A) + \Theta(B)$ as believed to be true
by Shannon himself \cite{Matousek}. However, for connection graphs, $\Theta$ is a compatible
with addition and multiplication:

\begin{coro}
If $A=\psi(G), B=\psi(H)$ are connection graphs, then 
$\Theta(A+B) = \Theta(A) + \Theta(B)$ and 
$\Theta(A \cdot B) = \Theta(A) \Theta(B)$. 
\end{coro}
\begin{proof}
This follows directly from the fact that for connection graphs as well as products
of connection graphs, the number $f_0(G)$ of $0$-dimensional elements in the complex is 
a ring homomorphism and that $\Theta(\psi(G))=f_0(G)$ for products of 
simplicial complexes. 
\end{proof}

\paragraph{}
The Shannon capacity joins now a larger and larger class of functionals
which are compatible with the ring structure: $f_0(G), |G|$, the Euler characteristic, 
Wu characteristic or the $\zeta$-function $\zeta(s) = \sum_j \lambda_j^{-s}$ 
defined by the eigenvalues $\lambda_j$ of the connection Laplacian $L=1+A(G)$ of $G$,
where $A(G)$ is the adjacency matrix of the connection graph $\psi(G)$. 

\section{Proof of Theorem 1}

\paragraph{}
Let $d(x)$ denote the {\bf vertex degree} of $x \in G$ when $x$ is considered to be 
a vertex in the connection graph $\psi(G)$. Let $\delta(x)$ denote the {\bf minimal 
vertex degree of all neighboring vertices} of $x$. Formally, this is
$$   \delta(x) = {\rm min}_{y \in S(x)} d(y) \; . $$
The following lemma shows that strict local minima of the function $d: V \to \mathbb{N}$
reveal the $0$-dimensional sets, the sets $x$ with cardinality $|x|=1$.

\begin{lemma}
If $G$ is a complex, then ${\rm dim}(x)=0$ for $x \in G$ if and 
only if $d(x)<\delta(x)$. 
\end{lemma}

\begin{proof}
Assume $x \in G$ is a $0$-dimensional point and assume that
$y \in G$ is connected to $x$. Because $x,y$ intersect and $x$ is a point, 
$x \subset y$. Furthermore, $y$ intersects any set $z$ which $x$ intersects
meaning $d(y) \geq d(x)$. 
Since $y$ also intersects some point different than $x$, we have $d(y) > d(x)$.
On the other hand, assume that $x$ is a point which is not $0$-dimensional.
Then it contains a $0$-dimensional point $y$ and $d(x) > d(y)$ by what
we have seen before. It therefore can not happen that the degree $d(x)$
is smaller than any neighboring degree.
\end{proof}

\paragraph{}
We can now prove Theorem 1: 
given a graph $\psi(G)$ we want to reconstruct the simplicial complex $G$.
By the Lemma we can identify the set $G_0$ of $0$-dimensional sets. This is 
an independent set already. None of them are adjacent because two different $0$-dimensional
sets do not intersect. The $1$-dimensional sets $G_1$ are the vertices $y$ in the graph which 
have the property that $y \in S(a) \cap S(b)$ where $a,b \in G_0$ are two different 
$0$-dimensional points. The $2$-dimensional sets $G_2$ are the vertices $y$ in the graph which 
have the property that $y \in S(a) \cap S(b) \cap S(c)$ with three different 
vertices $a,b,c$. We see that $G_k = \{ x \in V, x \in S(a_0) \cap S(a_2) \cdots \cap S(a_k)$,
$a_0,a_1,\dots, a_k \in G_0$ are all disjoint. 
This reconstruction only needs a polynomial amount of computation steps: we have to go through
all the $n$ vertices, compute $d(x)$ and then form intersections of unit spheres. 

\paragraph{}
The same proof also establishes that the Barycentric graph $\phi(G)$ 
determines $G$. The later could also be achieved also differently: we can see the {\bf facets}
of $G$ by looking at maximal sub-graphs which belong to Barycentric refinements of 
simplices. But this point of view is computationally much more costly as we have to 
find subgraphs which are refinements of simplices which in particular also means requires
to find complete subgraphs. 

\section{Proof of Theorem 2}

\paragraph{}
The proof of theorem 2 can serve as a nice independent introduction to 
graph homotopy. Doing graph homotopy steps can be seen as a game. Indeed, 
the {\bf homotopy puzzle} to deform a graph homotopic to $1$ to $1$ is a 
nice game. It can be difficult, like for dunce hats or bing houses. Like
for any game, it is good to know what combinations of moves can do. We will
see in particular that {\bf edge refinements} are homotopy steps. 

\paragraph{}
When combining two homotopy steps we can achieve that the set of vertices
does not change but that we can get rid of an edge. 

\begin{lemma}[Lamma A]
If $A$ is a graph and $e=(v,w)$ is an edge such that $S(v)$ and
$S(v)-w$ are both contractible, then $A$ is homotopic to $A-e$.
\end{lemma}

\begin{proof}
$A \to B=A-v$ is a homotopy step because $S(v)$ is contractible.
Now $B \to B +_A v$ is a homotopy step. And $B +_A v$ is $A-e$.
\end{proof}

\paragraph{}
If $e=(a,b)$ is an edge in a finite simple graph $A$, an 
{\bf edge refinement} $A+e$ is a new graph obtained from $A$ by 
replacing $e$ with a new vertex $e$
and connecting this new vertex to $a$, $b$ as well as every vertex in
$S(a) \cap S(b)$. Edge refinement is a homotopy. This is true in general for 
any graph. The graphs do not need to come from simplicial complexes.

\begin{lemma}[Lemma B]
If $G$ is a graph and $e=(a,b)$ is an edge. 
Then the edge refinement $G+e$ is homotopic to $G$.
\end{lemma}
\begin{proof} 
Without removing the edge $(a,b)$, add a new vertex $e$ and connect it to $a,b$ 
and $A=S(a) \cap S(b)$. 
Since the graph generated by $e,a,b,V(A)$ is now a unit ball with center $e$, 
it is contractible so that this is a homotopy step. 
Now, $S(a)$ and $S(a) \setminus b$ are both contractible in this new graph. 
By Lemma A, it remains contractible after removing the edge $e$ from it. 
Now we have the edge refinement. 
\end{proof}

\paragraph{}
The following Lemma is not true without the Barycentric assumption.

\begin{lemma}[Lemma C]
If $G$ is a Barycentric complex and $y,z$ are two elements with 
$y \cap z \neq \emptyset$, but not $y \subset z$ nor $z \subset y$, 
then $S(y) \cap S(z)$ is contractible in the connection graph of $G$. 
\end{lemma}

\begin{proof}
We know that $x=y \cap z$ is a simplex as an intersection of simplices.
Case (i): Assume there are no other parts
except subsets of $x$, then $x=S(y) \cap S(z)$ as a simplex is
contractible. 
Case (ii): If there is an other point $u$, we must have 
$u$ intersecting $x$ or then $u$ be contained in larger $v$ containing $x$.
Proof: Assume this is not the case, then we have $y,z,u$
which are pairwise not contained in each other but which intersect. Pick $0$-dimensional
points $x,a,b$ in the intersections. These points were already points in the original
complex $G$ from which one has taken the Barycentric refinement. 
The union $(x,a,b)$ of them generates a simplex $v$ in $G$ which is a point in 
the Barycentric refinement. This $v$ intersects all three sets $x,y,z$.
We can do that for any choice of $x,a,b$ so that there is a simplex $v$ containing $y,z,x$. 
So $S(y) \cap S(z)$ contains this point $v$ which is connected to all other points in 
$S(y) \cap S(z)$. So, $S(y) \cap S(z)$ is contractible.
\end{proof}

\paragraph{}
An example for a non-Barycentric complex, where it is false 
is $A = \psi(C_3)$, where $S(y \cap S(z)$ is not contractible. 

\paragraph{}
Here is the proof of Theorem 2: 

\begin{proof}
Assume $G$ is a Barycentric refined complex.
Let $y$,$z$ be two sets in $G$ which do intersect but which are not contained in each other. 
We want to show that removing this edge is a homotopy.                   
Using Lemma B, make an edge refinement with $e=(y,z)$. By Lemma C, $S(y) \cap S(z)$ is contractible
Now, $B=S(y) \cap S(z) + y + z$  is contractible. By Lemma A, we can remove the edge $e$ and have
a homotopy. Because the new vertex $e$ has a contractible unit sphere $S(y) \cap S(z) + y + z$, we can
remove it. Overall we have removed the edge $e$.
We can now do this construction for any connection between two sets which are not contained in each other.
In the end, we reach the graph $\phi(G)$. 
\end{proof} 

\paragraph{}
Going through connection graphs also allows to see that the Barycentric refinement
can be written as a homotopy: if $A$ is a graph, then its {\bf Barycentric
refinement} $A_1$ is the graph in which the simplices of $A$ are the
vertices and where two such simplices are connected if one is contained
in the other.

\begin{coro}
Any graph $A$ and its Barycentric refined graph $A_1$ are homotopic. 
\end{coro}
\begin{proof} 
The deformation from $A$ to $A_1$ can be 
done by edge refinements (Lemma B). First remove the vertices belonging
to the highest dimensional simplices, then get to the next
smaller points and edge refine them, until only the vertices of $A$ 
are left as points.
\end{proof} 

\begin{figure}[!htpb]
\scalebox{1.0}{\includegraphics{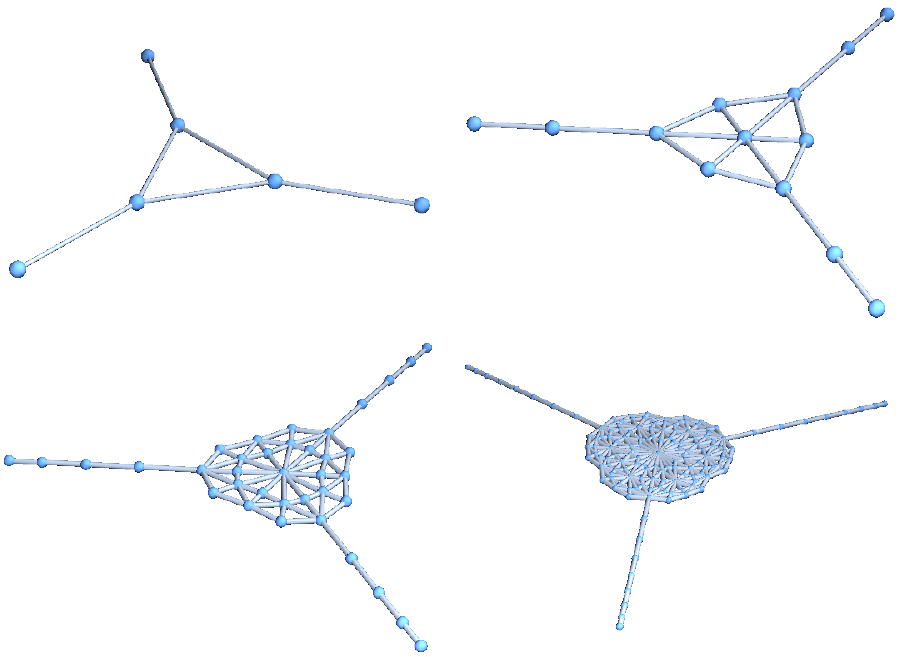}}
\label{Figure 1}
\caption{
we see three Barycentric refinements. Each Barycentric refinement
$A \to A_1$ is a homotopy but we can not contract $A_1$ to $A$ 
in general directly. We need to use both contractions and 
expansions to do that. 
}
\end{figure}

\section{Proof of Theorem 3}

\paragraph{}
Theorem 3 leaves finite combinatorics and looks at the continuum.
If $A=(V,E)$ is a finite simple graph with $|V|=n$ vertices we 
can embed it into $\mathbb{R}^{n}$. Put the vertices $x_k$ as
unit vectors $e_k$ in $\mathbb{R}^n$. Now connect each pair
$(a,b)$ of vertices which are connected by a line segment. 
Then look at all $3$-cliques $(a,b,c)$ of point triples which 
are all pairwise connected. This produces a concrete triangle
spanned by the points $e_a,e_b,e_c$ in $\mathbb{R}^n$. Go on 
like this with all $k$-cliques. The realization $\gamma(A)$
produces a compact subset of $\mathbb{R}^n$ equipped with
the Euclidean norm. We usually con realize a complex in much lower
dimension.

\paragraph{}
Two functions $f_0:X \to X$ and $f_1:X \to X$ are {\bf homotopic}
if there exists a continuous map $F: X \times [0,1] \to X$
such that $F(x,0)=f_0(x)$ and $F(x,1) = f_1(x)$. 
Two topological spaces $X,Y$ are {\bf homotopic} if there
is a pair of continuous maps $f:X \to Y$ and $g: Y \to X$ such
that $g \circ f: X \to X$ is homotopic to the identity map $I(x)=x$.

\paragraph{}
The construction of a geometric realization of two homotopic graphs $A$,$B$
produces two topological spaces $\gamma(A),\gamma(B)$ which are
classically homotopic. 

\paragraph{}
To perform the proof, one only has to check this for a single homotopy extension or its inverse. We proceed by induction.
Assume the graph $A$ has already been realized as $\gamma(A) \subset \mathbb{R}^n$. Now add a new vertex
by attaching it to a contractible subgraph $C$ of $A$. Take the larger space $\mathbb{R}^n \times \mathbb{R} = \mathbb{R}^{n+1}$
and place the new vertex there. Now build the pyramid over $\gamma(C)$. This produces an explicit embedding
of the larger graph. The homotopy $F_t(x,h) \to (x,t h)$ deforms the geometric realization of
$\gamma(A)$ to $\gamma(A+x)$.

\section{Proof of Theorem 4}

\paragraph{}
It is enough to prove that if $A'$ is a homotopy extension of $A$
and $B$ is fixed, then $A'  \cdot  B$ is a homotopy extension of $A  \cdot  B$. 
The general case can be done by two such steps, one for $A$ and one for $B$. \\
Let $x$ be a new vertex so that $A'= A +_C x$ with a contractible 
subgraph $C$ of $A$. If $V$ is the vertex set of $A$, let $W$ the vertex
set of $B$ and let $V' = V \cup \{x\}$ the vertex set of $A'$. 
The Cartesian product $V \times W$ is the vertex set of $A  \cdot  B$
and $V' \times W$ is the vertex set of $A'  \cdot  B$. \\
There are $m=|W|$ copies $x_1, \dots, x_m$ 
of the vertex $x$ in $A' \cdot  B$. The graph $A' \cdot  B$
is a $|W|$ fold homotopy extension of $A  \cdot  B$: start with $A  \cdot  B$, 
then add $x_1$ and attach it to the graph $U_1$ generated by the union 
of all sets $C_y = C \times \{y\} \subset A \cdot  B$ and vertices $x \in V(A)$ 
connected to $x_1$ and summing over all $y \in V(B)$ with $(x,y) \in E(B)$. \\
Continue like this until all $x_k$ are added. This works, because each 
$U_k$ is contractible.  At the end we have $A' \cdot  B$. 

\paragraph{}
The analogue result of Theorem 4 for addition (disjoint union) is clear. 
If $A$ is homotopic to $A'$ and $B$ is homotopic to $B'$ then the 
there is a deformation of $A+B$ to $A'+B'$. 

\section{Proof of Theorem 5}

\paragraph{}
If $G,H$ be finite abstract simplicial complexes and let $\psi(G),\psi(H)$
their connection graphs. To the product $G \times H$ belongs the
strong product graph $\psi(G \times H) = \psi(G) \cdot \psi(H)$. The vertices of this graph is
the Cartesian product $G \times H$ as a set of sets. Two points 
$(a,b)$ and $(c,d)$ in this product are connected if $a \cap c \neq \emptyset$
and $b \cap d \neq \emptyset$. Let $G_0$ denote part of the vertex set $G$
of $\psi(G)$ consisting of $0$-dimensional sets. 

\begin{lemma}
For any $G$, we have $i(\psi(G))=|G_0|$. 
\end{lemma}

\begin{proof}
The set $\{ x \in G,  {\rm dim}(x)=0 \}$ is an independent set of 
vertices in $\psi(G)$. The reason is that two zero dimensional sets
do not intersect. This shows $i(A) \leq |G|$. But assume we have an 
independent set $I$ which contains a positive dimensional vertex
$x=(a_1,\dots,a_k)$. But then we can replace $x$ with a larger
independent set $\{a_1, \dots, a_k \}$. 
\end{proof}

\paragraph{}
The independence property holds also for the product: 

\begin{lemma}
if $A=\psi(G),B=\psi(H)$ are connection graphs, then 
$i(A \cdot B) = i(A) i(B) = f_0(G) f_0(H)$. 
\end{lemma}

\begin{proof}
The $0$-dimensional parts $G \times H$ are the points $(x,y)$,
where both $x,y$ are zero dimensional. There are $f_0(G) f_0(H)$
such sets  in $G \times H$ and they form an independent set in 
$\psi(G \times H)$. For the same reason than in the previous lemma,
we can not have an independent set containing one positive dimensional
point because we could just replace it with the collection of its 
zero dimensional parts which are all independent of each other. 
\end{proof} 

\paragraph{}
We can write the Shannon capacity as
$$ \Theta(A) = \lim_{n \to \infty} (i(A^n))^{1/n} \; . $$
It is bound below by $i(A)$ and bound above by the {\bf Lovasz number}
$\theta(A)$. 

\paragraph{}
As part of the {\bf sandwich theorem}, we have have $i(G) \leq \Theta(G) \leq \theta(G)$. 
In our case, we can see $\Theta(\psi(G)) \geq f_0(G) = i(G)$ also
because the zero-dimensional parts remain also in the product an independent set.
But also in the product $\psi(G^k)$ if we have a product point $(a_1,\dots,a_k)$
in the independent set, where some $a_j$ is not zero dimensional, then we can
replace this point by individual points and so increase the independence number.
This shows that $\Theta(\phi(G))=\Theta(\psi(G))=f_0(G)$.
Applying the previous lemma again and again, we have $i(G^n) = f_0(G)^n$
so that $\Theta(G)=f_0(G)=m$. \\

\paragraph{}
An elegant proof uses the {\bf Lovasz umbrella construction} which attaches 
to every vertex $x$ in the graph a unit vector $u(x)$ such that 
the dot product $\langle u(x) \cdot u(y) \rangle=0$ if $(x,y)$ is not an edge. In our
case, this means that the vectors $u(x),u(y)$ are orthogonal if $x \cap y = \emptyset$. 
The explicit {\bf Lovasz representation} in $\mathbb{R}^{f_0}$ is
$$ u_j(x) = \left\{ \begin{array}{lr} 1 , & j \in x \\
                                      0 , & {\rm else} \end{array} \right. \; . $$
It is clear that $u(x) \cdot u(y) =0$ if $x \cap y = \emptyset$. 
The {\bf stick of the Umbrella} is the vector $c=[1,1,\dots,1]/\sqrt{f_0}$. 
The Lovasz number is bound above by any value
$$  {\rm max}_{x \in G} (u(x) \cdot c)^{-2} \;  $$
of a choice of a Lovasz umbrella and stick (it is the $\sec^2(\alpha)$ of the maximal
opening angle of the umbrella) and for our Lovasz umbrella equal to $f_0$. 
So, we know $\theta(\psi(G)) \leq f_0$. But we also have the lower bound $i(\psi(G)) = f_0$
so that $\Theta(\psi(G)) = f_0$. 

\paragraph{}
We have $\Theta(\phi(G)) \geq f_0(G)$:
since $\phi(G)$ is a subgraph of $\psi(G)$, we 
have $\Theta(\phi(G)) \geq \Theta(\psi(G)) = f_0(G)$.
We have $\Theta(\psi(G)) = f_0(G)$ we have in general 
$\Theta(\phi(G)) < \theta(\phi(G))$.

\paragraph{}
In some cases, we can get an umbrella for $\phi(G)$ which lives in $\mathbb{R}^{f_0(G)}$ and so
get $\Theta(\phi(G)) = f_0(G)$. For $G=K_3$, where $\phi(G)$ is a wheel graph with 6 spikes, 
we have $i(G)=f_0(G)=3$. The umbrella in $\mathbb{R}^3$ has the stick $[1,1,1]/\sqrt{3}$ and
assigns vectors $u(x) = {\rm min}_{v \in x} e_v$. The is the standard basis. In general for 
the graph $C_{2n}$ or its pyramid extension $W_{2n}$, the wheel graph with boundary $C_{2n}$, 
we have the Lovasz system $u( \{1\}) = u( \{1,2\} ) = e_1$, $u(\{2\}) = u(\{2,3\} ) = e_2$
etc $u( \{n\}) - u( \{n,1\} ) = e_n$ in the $C_{2n}$ case and 
$u( \{1\}) = u( \{1,2\} ) = u( \{1,2,n+1 \})  = e_1$, $u( \{2\}) = u( \{2,3\} ) = u( \{2,3,n+1 \}) = e_2$
etc $u( \{n\}) - u( \{n,1\} ) = u( \{n,1,n+1 \}) = e_n$ in the wheel graph 
case. Now $W_{6} = \phi(K_3)$. 

\paragraph{}
Here is an example of $\phi(G)$ without a Lovasz umbrella. For the Barycentric refinement of the 
{\bf figure $8$ complex} 
$$  G=\{ \{1,2,3,4,5,6,7,(1,2),(2,3),(3,4),(4,1),(4,5),(5,6),(6,7),(7,4) \}) $$ 
which is a genus 
$b_1=2$ complex with Euler characteristic $b_0-b_1=-1$ (equivalently by Euler-Poincar\'e $f_0-f_1=7-8=-1$), we would
have to attach to every edge a unit vector and since none of the edges connect in the connection graph, we 
would need so $8$ pairwise perpendicular vectors. This can not be done in $\mathbb{R}^{f_0}$. We see that
$\theta(\phi(G)) > 7$. Shannon computed the capacity for all graphs with $\leq 6$ vertices except for $C_5$,
where Lovasz eventually established $\sqrt{5}$. Here is a general observation:

\paragraph{}
If $G$ is a one-dimensional simplicial complex with $f_0$ vertices and $f_1$ edges,
then the graph $\phi(G)$ has the Shannon capacity ${\rm max}(f_0,f_1)$ because
both the edge set as well as the vertex set is an independent set in $\phi(G)$. 
For the figure 8 complex above we have $\Theta(\phi(G))=8$ and $\Theta(\psi(G))=7$. 
When looking at bouquets of circles, we can arbitrary large differences between $\Theta(\phi(G))$
and $\Theta(\psi(G))$.  

\begin{figure}[!htpb]
\scalebox{1.0}{\includegraphics{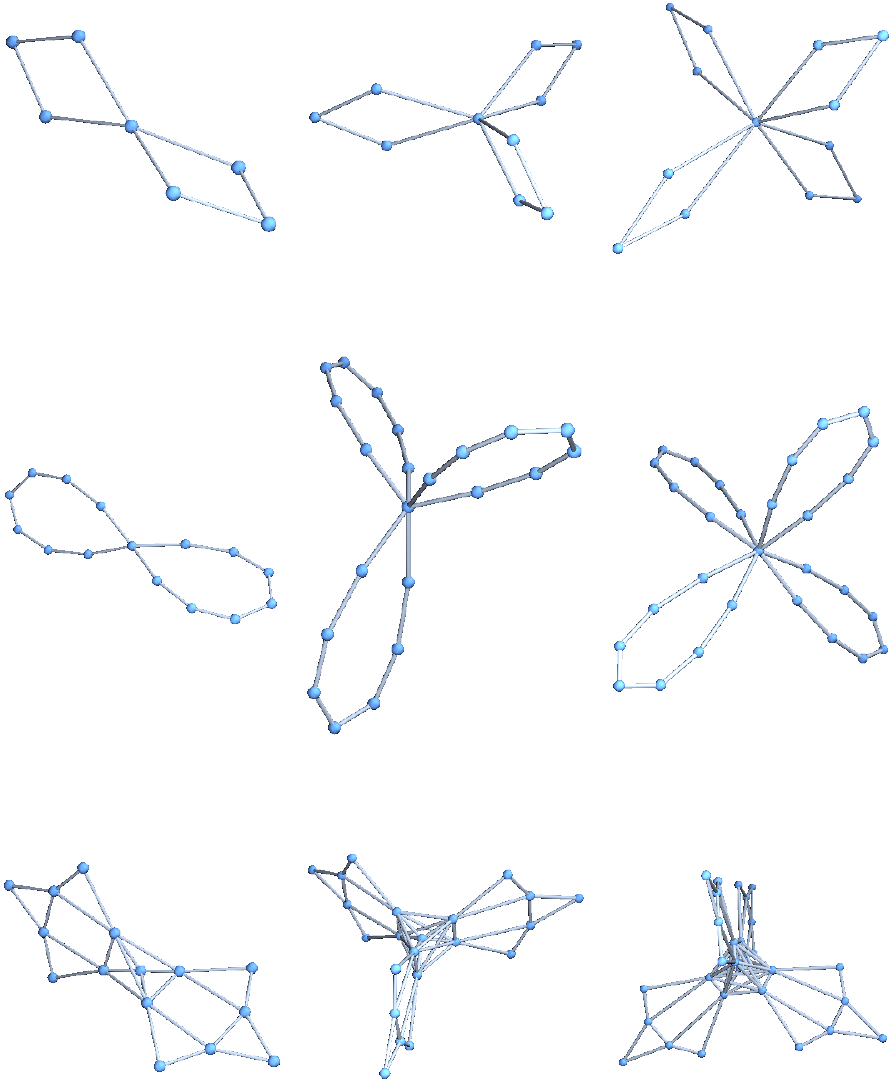}}
\label{Figure 16}
\caption{
Bouquet complexes $G_2,G_3,G_4$ of genus $k=2,3,4$ its Barycentric 
graphs $\phi(G_k)$ and its connection graphs $\psi(G_k)$. 
We have $\Theta(\phi(G_2))=f_1(G_2)=8, \Theta(\psi(G_2))=f_0(G_2)=7$,
and $\Theta(\phi(G_3))=f_1(G_3)=12, \Theta(\psi(G_3)=f_0(G_3)=10$,
and $\Theta(\phi(G_4))=f_1(G_3)=16, \Theta(\psi(G_4)=f_0(G_3)=13$,
}
\end{figure}

\section{Homotopy game}

\paragraph{}
The concept of homotopy for graphs produces
puzzles: give two graphs which are homotopic and solve the task 
to find a concrete set of homotopy steps which move one to the other. A simple
example is to deform $C_5$ to $C_6$ or to deform the octahedron to the icosahedron. 
In general, one can not go with homotopy reduction steps alone but needs
first to enlarge the dimension. The reason is simple: for a discrete manifold, each 
unit sphere is sphere and not homotopic to $1$. One can not apply a single homotopy
contraction step on a discrete manifold. One can do more complicated moves like edge
reductions, covered in Lemma B.

\begin{figure}[!htpb]
\scalebox{1.0}{\includegraphics{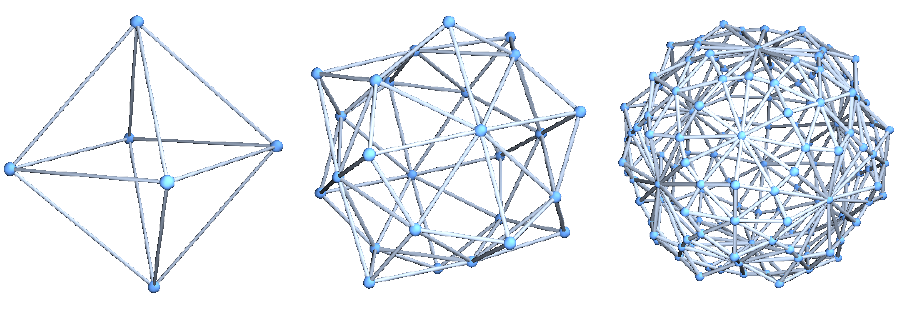}}
\label{Figure 5}
\caption{
The octahedron graph and the first two Barycentric refinements.
}
\end{figure}

\begin{figure}[!htpb]
\scalebox{1.0}{\includegraphics{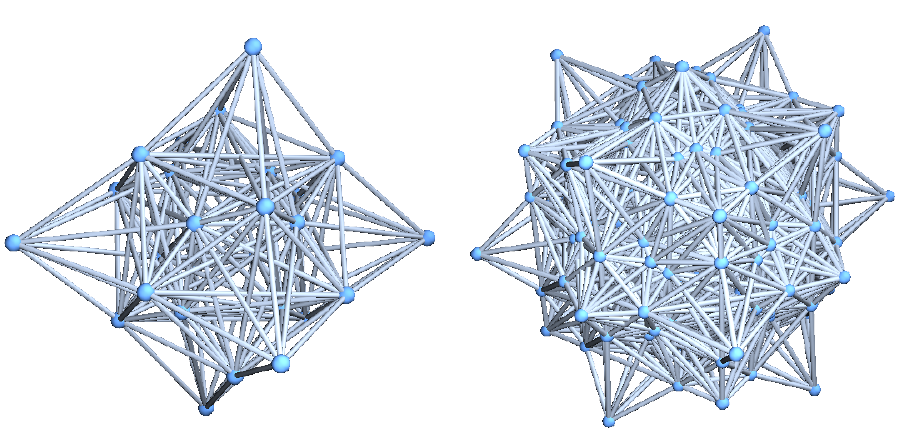}}
\label{Figure 6}
\caption{
The connection graphs of the octahedron graph $O$ and the first
Barycentric refinement complex $O_1$. While $\psi(O)$ is homotopic to a $3$-sphere
(we see this by brute force computing the Betti numbers $(1,0,0,1)$),
the graph $\psi(O_1)$ is homotopic to the $2$-sphere with Betti numbers $(1,0,1)$. 
}
\end{figure}

\paragraph{}
The homotopy puzzle leads also to 2-player games: player Ava
proposes a graph, player Bob must reduce it to a point. If Bob succeeds, he
wins. If player Bob does not succeed, Ava has to reduce it to a point. If
she does not succeed, player Bob wins. In the next round, the roles of player Ava
and Bob are reversed. The game has a creative aspect and has the advantage that
it can be played on any level: two small kids can be challenged similarly
as two graph theory specialists. 
One can play more advanced versions by giving two homotopic graphs and ask
the opponent to deform one into an other. For example, one can ask to deform 
an octahedron to an icosahedron.  

\begin{figure}[!htpb]
\scalebox{1.0}{\includegraphics{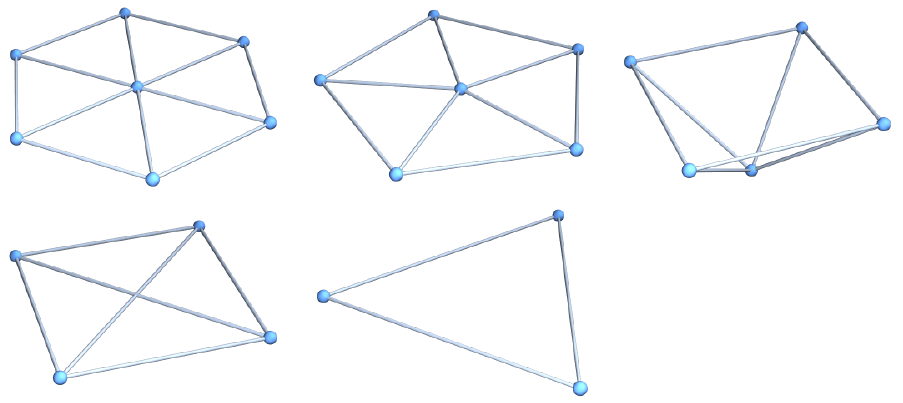}}
\label{Figure 7}
\caption{
Barycentric refinement is a homotopy deformation. The 
refinement of a triangle can be done with first expanding to a 
tetrahedron, then do three edge refinements.
}
\end{figure}

\section{General remarks}

\paragraph{}
{\bf Categorical view}.
Simplicial complexes form a category in which the complexes are
the objects and the order preserving maps are the morphisms. 
A {\bf partial order} on a complex is given by incidence $x \subset y$ of its sets. 
The sets in $G$ are called simplices or faces. The complete complex 
$G=K_{n+1}=[0,\dots,n]$ can naturally be identified with the n-dimensional face
$x=\{ 0,\dots,n \}$ which is a facet in $G$, a largest element in the partial order.
Also finite simple graphs for a category in which the graph homomorphisms are
the morphisms: $f:(V,E) \to (W,F)$ is a homomorphism if $f(V)=W$ and if $(a,b) \in E$,
then $(f(a),f(b)) in F$. One often sees graphs as one-dimensional simplicial complexes.
A more natural functor is to assigns to a graph its Whitney complex.

\paragraph{}
{\bf The continuum}.
We usually do not look at the geometric realization functor $\gamma$ to topological spaces
because this leaves combinatorics and requires stronger axiom systems.
It can not hurt to compare classical homotopy with discrete graph homotopy,
in particular because many textbooks treat graphs as one-dimensional simplicial complexes.
Mathematicians like Euler, Poincar\'e or Alexandroff \cite{alexandroff} considered graphs
building higher dimensional structures (Geruest stands for scaffold). 
Seeing them as 1-dimensional skeleton simplicial complexes came later.
Later, (like in \cite{Ivashchenko1993}) the language of graph theory
was again recognized as valuable for low dimensional topology. 
It is close to \cite{WhiteheadI} but when simplified as in \cite{CYY},
the notion becomes even more accessible. 

\paragraph{}
{\bf Genus spectrum}.
We have seen in a \cite{Spheregeometry} that Barycentric refinements ``smooth out" pathologies.
We had seen there that the set of possible genus values $\{ 1-\chi(S(x)) \}_{x \in G}$ 
is stable after one Barycentric refinement. It is therefore a combinatorial invariant of $G$,
(which by definition is an invariant which does not change any more when doing 
refinements). For discrete even dimensional manifolds the genus spectrum is $\{ 1 \}$
and for odd-dimensional manifolds, the sphere spectrum is $\{ -1 \}$. The sphere
spectrum can be more interesting for discrete varieties. For the figure $8$ complex for
example, it is $\{ -1,-3 \}$. We have seen here an other instance where Barycentric refinement
makes Shannon capacity computable. 

\paragraph{}
{\bf Riemann Hurwitz}.
Barycentric refined complexes are also needed when looking at a Riemann-Hurwitz
theorem. If $G$ is a finite simple graph which is a Barycentric refinement
and $A$ is a finite group of order $n$ acting by automorphisms, then $\phi(G)$ is a 
ramified cover over the graph $H=G/A$.
The Riemann-Hurwitz formula is then $\chi(G) = n \chi(G/A) - \sum_{x \in G} e_x$, where
$e_x=\sum_{a \neq 1, a(x)=x} (-1)^{{\rm dim}(x)}$ is a {\bf ramification index}. This can be derived
from \cite{brouwergraph}. The graph $G$ is a cover of $G/A$. This cover is {\bf unramified} 
if the ramification indices $e_x$ are all zero everywhere. 
In this case, the cover $G \to H$ is a discrete {\bf fibre bundle} with structure group $A$.

\paragraph{}
{\bf Finitism}. 
Given any laboratory X, we can only do finitely many experiments
each only accessing finitely many data. Every measurement defines
a partition of X into sets of experiments which can not be distinguished
by f. A finite set of functions generates a simplicial complex G if the 
smallest sets are collapsed to point. 
One can see this also as follows. Let $(X,d)$ is a compact metric space modelling the 
laboratory. Take $\epsilon>0$ and take a finite set $V$ of points such that every $x \in X$
is $\epsilon$ close to a point in $V$. A non-empty set of points is a simplex if all
points are pairwise closer than $\epsilon$. The set of all simplices is a
simplicial complex $G$. 

\paragraph{}
{\bf $\pi$-Systems}.
One can also look at the category of set of non-empty sets which are closed under the operation 
of non-empty intersection. Any such structure $G$ is homomorphic to a simplicial complex: 
just collapse every atom to a point and remove any element which does not effect
the order relation. If we allow also also the empty set, then we have a $\pi$-system
a set of sets which is closed under the operation of taking finite intersections. 
A $\pi$-system without empty set can be generalized to a {\bf filter base} in which the 
intersection of two elements must contain an element in the set. 
Every $\pi$-system is isomorphic to a simplicial complex but the super symmetry structure
given by the sign of the cardinality is not compatible. It produces inconsistent Euler
characteristic for example. Euler characteristic is an invariant under isomorphisms of simplicial
complexes but not an isomorphism invariant for $\pi$-systems. 

\paragraph{}
{\bf The SM theorem}.
By the Szpilrajn-Marczewski theorem \cite{Szipilrajn-Marczewski}, every graph 
can be represented as a connection graph of a set of sets. This theorem is also abbreviated as 
SM theorem in intersection graph theory. Edward Szpilrajn-Marcewski (1907-1976) proved this
in 1945. The theorem has been improved and placed into extremal graph theory by
Erd\"os, Goodman and Posa who showed in 1964 \cite{ErdoesGoodmanPosa}
that one can realize any
graph of $n$ vertices as a set of subsets of a set $V$ with $[n^2/4]$ elements and that
for $n \geq 4$, the set of sets can be all different.

\paragraph{}
{\bf Spectral question}.
We still do not know whether 
the spectrum of the connection Laplacian determines the complex $G$.
We know that the number of positive eigenvalues is the number of
even dimensional simplices. We know also that the spectrum $\sigma(L(\psi(G \times H)))$
of the connection Laplacian of $\psi(G \times H)$ has eigenvalues $\lambda_j \mu_k$
where $\lambda_j$ are the eigenvalues of $\psi(G)$ and $\mu(k)$ are the eigenvalues
of $\psi(H)$. The spectrum of connection Laplacians could contain more
topological information.

\paragraph{}
{\bf Characterize SM graphs}.
Connection graphs are special as they can be realized by on sets with $n$ elements
This begs for the question which graphs with n vertices have the property that they can
be realized with a set $G$ of subsets of a set with $n$ elements. 
The non-simplicial complex example $G=\{\{1,2\},\{2,3\},\{3,1\}\}$ with three sets
belongs to a graph $C_3$ that can be realized on a set with $n=3$ atoms. 

\section*{Appendix: Euler's Gem}

\paragraph{}
In this appendix, we give a combinatorial proof of the Euler gem formula telling that a 
$d$-sphere has Euler characteristic $1+(-1)^d$. We also classify {\bf Platonic d-spheres}.
The first part of this document was handed out on February 6, 2018 at
a Math table talk ``Polishing Euler's gem", a day before Euler's day 2/7/18. The section
is pretty independent of the previous part and added because it gives more details about 
homotopy and should be archived somewhere. 

\paragraph{}
A {\bf finite simple graph} $G=(V,E)$ consists of two finite sets, the {\bf vertex set} 
$V$ and the {\bf edge set} $E$ which is a subset of all sets $e=\{a,b\} \subset V$ with
cardinality two. A graph is also called a {\bf network}, the vertices are the {\bf nodes} and
the edges are the {\bf connections}. A subset $W$ of $V$ {\bf generates} a subgraph $(W,F)$ 
of $G$, where $F=\{ \{a,b\} \in E \; | \; a,b \in W \}$. 
Given $G$ and $x \in V$, its {\bf unit sphere} is the sub graph generated by
$S(x) = \{ y \in V | \{x,v\} \in E \}$. The {\bf unit ball} is the sub graph generated by
$B(x) = \{ x\} \cup S(x)$. Given a vertex $x \in V$, the graph $G-x$ {\bf with $x$ removed}
is generated by $V \setminus \{x\}$. We can identify $W \subset V$ with the subgraph
it generates in $G$.  

\paragraph{}
The empty graph $0=(\emptyset,\emptyset)$ is the {\bf $(-1)$-sphere}.
The $1$-point graph $1=(\{1\},\emptyset)=K_1$ is the smallest contractible graph.
Inductively, a graph $G$ is {\bf contractible}, if it is either $1$ or if 
there exists $x \in V$ such that both $G-x$ and $S(x)$ are contractible. As seen by induction,
all {\bf complete graphs} $K_n$ and all trees are contractible. Complete subgraphs are also called
{\bf simplices}. Inductively a graph $G$ is called a {\bf $d$-sphere}, if it is either $0$ or
if every $S(x)$ is a $(d-1)$-sphere and if there exists a vertex $x$ such that $G-x$ is contractible. 

\paragraph{}
Let $f_k$ denote the number of complete subgraphs $K_{k+1}$
of $G$. The vector $(v_0,v_1, \dots)$ is the {\bf $f$-vector} of $G$ and 
$\chi(G)=v_0-v_1+v_2- \dots $ is the {\bf Euler characteristic} of $G$. 
Here is Euler's gem: 

\begin{thm}
If $G$ is a $d$-sphere, then $\chi(G) = 1+(-1)^d$. 
\end{thm}

\paragraph{}
To prove this, we formulate three lemmas. Given two subgraphs $A,B$ of $G$, the 
{\bf intersection} $A \cap B$ as well as the {\bf union} $A \cup B$ are sub-graphs of $G$.

\begin{lemma}
$\chi(A)+\chi(B)=\chi(A \cap B) + \chi(A \cup B)$ for any two subgraphs $A,B$ of $G$. 
\end{lemma} 
\begin{proof} 
Each of the functions $f_k(A)$ counting the number of $k$-dimensional simplices
in a subgraph $A$ satisfies the identity. The Euler characteristic $\chi(G)$ 
is a linear combination of such valuations $f_k(G)$ and therefore satisfies the identity.
\end{proof} 

\paragraph{}
A graph $G$ is a {\bf unit ball}, if there exists a vertex $x$ in $G$ such that 
the graph generated by all points in distance $\leq 1$ from $x$ is $G$. We write 
$B(x)$ and rephrase it that it is a {\bf cone extension} of the unit sphere $S(x)$.
Algebraically, one can say $B(x) = S(x) + 1$ where $+$ is the join addition. 

\begin{lemma}
Every unit ball $B$ is contractible and has $\chi(B)=1$. 
\end{lemma} 
\begin{proof} 
Use induction with respect to the number of vertices in $B=B(x)$. It is true
for the one point graph $G=K_1$. Induction step: given a unit ball $B(x)$. Pick $y \in S(x)$ for which $B(y)$ is not
equal to $B(x)$ (if there is none, then $B(x)=K_n$ for some $n$ and $B(x)$ is contractible
with Euler characteristic $1$). Now, both $B(y), B(y)-x$ and $S(x)$ are smaller balls 
so that by induction, all are contractible with Euler characteristic $1$.
As both $B(x) \setminus y$ and $S(y)$ are contractible, also $B(x)$ is 
contractible. By the valuation formula, $\chi(B(x)) = \chi(B(x)-y) + \chi(B(y)) - \chi(S(x))=1+1-1=1$. 
\end{proof}

\begin{lemma}
If $G$ is contractible then $\chi(G)=1$.
\end{lemma}
\begin{proof}
Pick $x \in V$ for which $S(x)$ and $G-x$ are both contractible.
By induction, $\chi(G-x)=1$ and $\chi(S(x))=1$. By the unit ball lemma, $\chi(B(x))=1$. 
By the valuation lemma, $\chi(G)=\chi(B(x)) + \chi(G-x) - \chi(S(x))=1+1-1$. 
\end{proof} 

\paragraph{} 
Here is the proof of the Euler gem theorem. 

\begin{proof}
For $G=0$ we have $\chi(G)=0$. This is the induction assumption.
Assume the formula holds for all $d$-spheres.
Take a $(d+1)$-sphere $G$ and pick a vertex $x$ for which both $S(x)$ and $G-x$ are 
contractible. Now, $\chi(G)=\chi(G-x)+\chi(B(x))-\chi(S(x))=1+1-(1+(-1)^d)=1+(-1)^{d+1}$. 
\end{proof}

\paragraph{}
We look now at regular (Platonic) $d$-spheres. For $d=2$, we miss the tetrahedron (because 
this is a 3-dimensional simplex $K_4$) as well as the cube and dodecahedron because their Whitney complexes
are one-dimensional. 
Combinatorially, one can include them using CW complex definitions.
The classification of {\bf Platonic d-spheres} is very simple. First to the {\bf recursive definition}:
a {\bf Platonic d-sphere} is a $d$-sphere for which all
unit spheres are isomorphic to a fixed Platonic $(d-1)$-sphere $H$.
The {\bf curvature} $K(x)$ of a vertex is defined as $K(x)=\sum_{k=0} (-1)^k f_k(x)/(k+1)$,
where $f_k(x)$ is the number of $k$-dimensional simplices $z$ which contain $x$. 

\begin{lemma}[Gauss-Bonnet] $\sum_{x \in V} K(x) = \chi(G)$. 
\end{lemma}
\begin{proof}
We think of $\omega(x) = (-1)^{{\rm dim}(x)}$ as a charge attached to a set $x$. 
By definition, $\chi(G) = \sum_{x \in G} \omega(x)$. Now distribute the 
charge $\omega(x)$ from $x$ equally to all zero dimensional parts containing 
in $x$. 
\end{proof}

\paragraph{}
For a triangle-free graph, $K(x)=v_0-v_1/2=1-{\rm deg}(x)/2$. 
For a $2$-graph and in particularly for a $2$-sphere, 
we have $K(x)=v_0-v_1/2+v_2/3=1-{\rm deg}(x)/6$, where ${\rm deg}(x)$
is the {\bf vertex degree}.

\begin{thm}
There exists exactly one Platonic $d$-sphere except for $d=1,d=2$ and $d=3$. 
For $d=1$ there are infinitely many, for $d=2$ and $d=3$ there are two. 
\end{thm}

\begin{proof}
$d=-1,0,1$ are clear. For $d=2$, the curvature $K(x)=1-V_0/2 + V_1/3-V_2/3$ is 
constant, adding up to $2$. It is either $1/3$ or $1/6$. For $d=3$, where 
each $S(x)$ must be either the octahedron or icosahedron, $G$ is the
$16$-cell or $600$-cell. For $d=4$, by Gauss-Bonnet, $K(x)$ add up to $2$ and be of the form $L/12$. 
For $L=1$, there exists the $4$-dimensional cross-polytop with $f$-vector $(10, 40, 80, 80, 32)$.
There is no $4$-sphere, for which $S(x)$ is the $600$-cell as the $f$-vector of it is $(120,720,1200,600)$.
We would get $K(x)=1-120/2+720/3-1200/4+600/5 =1$ requiring $|V|=2$ and ${\rm dim}(G) \leq 1$.
\end{proof}

\begin{figure}[!htpb]
\scalebox{1.0}{\includegraphics{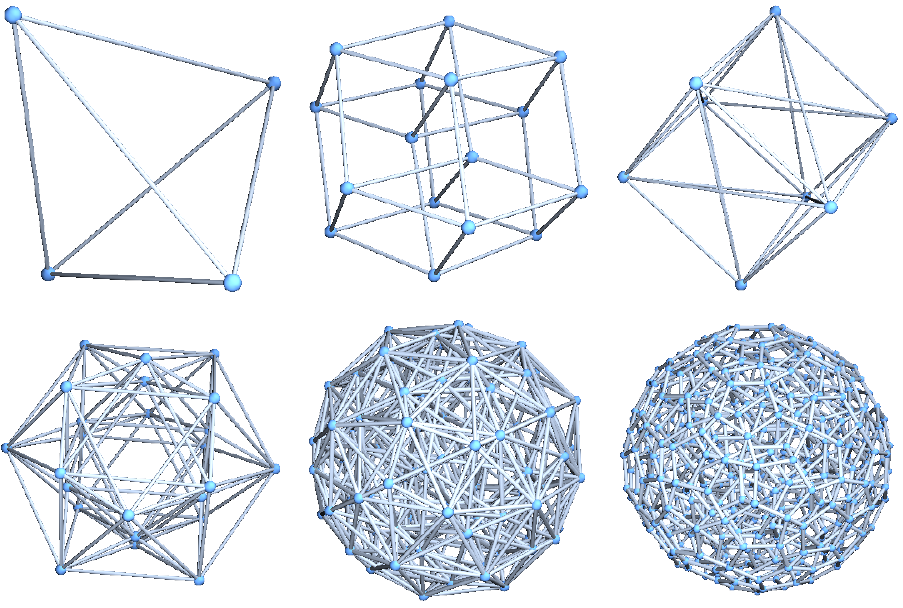}}
\label{Figure 14}
\caption{
The six classical $4$-polytopes are the $5$-cell, $8$-cell, $16$-cell, $24$-cell, the $600$-cell and 
$120$-cell. Only the $16$-cell and the $600$-cell are $3$-spheres, 
discrete $3$-dimensional discrete manifold.  
}
\end{figure}

\paragraph{}
In order to include all classically defined Platonic solids (including also the
dodecahedron and cube in $d=2$), we would need to use
{\bf discrete CW complexes}. Recursively, a {\bf $k$-cell} requires to identify a $(k-1)$ sub sphere in the
already constructed part. A CW-complex defines the Barycentric graph, where the cells are
the vertex set and where two cells are connected if one is contained in the other. 
A CW complex is {\bf contractible} if its graph is contractible. A CW complex $G$ is a {\bf d-sphere}
if its Barycentric graph $\phi(G)$ is a {\bf d-sphere}.
Each cell in a complex has a dimension ${\rm dim}(x)$, which is one more than the dimension
of the sphere, the cell has been attached to.  One can then define a {\bf Platonic d-sphere}, 
to have the property that for every dimension $k$, there exists a Platonic CW $k$-sphere $H_k$ 
such that for every cell of dimension $k$, the unit sphere is isomorphic to $H_k$. 

\paragraph{} The {\bf Zykov sum} \cite{Zykov} or {\bf join} of two graphs $G=(V,E),H=(W,F)$ 
is the graph $G+H=(V \cup W, E \cup F \cup \{ \{a,b\} | \; a \in V, b \in W \}$. 
For example, the Zykov sum of $S_0 + S_0 = C_4$. And $C_4 + S_0 = O$ is the octahedron graph. 
The sum of $G$ with the $0$-sphere $S_0$ is called the {\bf suspension}.
The sum of $G$ with $1=K_1$ is a {\bf cone extension} and by definition always a ball. 
One can quickly see from the definition that under taking graph complements, the 
{\bf Zykov-Sabidussi ring} with Zykov join as addition and large multiplication is dual to the 
{\bf Shannon} ring, where the addition is the disjoint union and where the multiplication 
is the strong product. 

\begin{lemma}
If $H$ is contractible and $K$ arbitrary then the join $H+K$ is contractible.
\end{lemma}
\begin{proof}
Use induction with respect to then number of vertices in $H$. It is clear for $H=K_1$. 
In general, there is a vertex in $H$ for which $S(x)$ is contractible. As $S(x)+K$ is 
contractible also $(y+S(x)) + K = y + (S(x) + K)$ contractible. 
\end{proof} 

\begin{lemma}
The join $G$ of two spheres $H+K$ is a sphere. 
\end{lemma} 
\begin{proof}
Use induction. One can use the fact that for $x \in V(H)$, we have 
$S_{H+K}(x) = S_H(x) + K$ which is a sphere and 
for $x \in V(K)$ one has $S_{H+K}(x) = H + S_K(x)$.  This shows that all 
unit spheres are spheres. Furthermore, we have to show that when removing
one vertex, we get a contractible graph. For $x \in V(H)$ one has then a 
sum of a contractible $H-x$ and $K$, for $x \in V(K)$ one has 
a sum of $H$ and $K-x$ which are both contractible. \\
For a suspension $G \to G+P_2$, it is more direct. Use induction:
if $G=H + \{a,b\}$ with $H=(V,E)$ then every unit sphere of $x \in V(H)$ becomes
after suspension a unit sphere of $G$. The unit spheres $S(a)=S(b)=H$ are already 
spheres. If we take away $a$ from $G$, then we have $G-a=H + \{b\}$ which is a ball
and so contractible. In general: either take a away a vertex $x$ of $H$ or from $K$. 
Now $H-x$ is contractible and so $(H-x) + K$ contractible by the previous lemma. 
That means $G-x$ is contractible. 
\end{proof} 

\begin{lemma}
The genus $j(G)=1-\chi(G)$ satisfies $j(H) j(G) = j(H+G)$. 
\end{lemma}

The set of spheres with join operation is a {\bf monoid} with zero element $0$. 
The {\bf graph complement} of $G=(V,E)$ with $|V|=n$ is the graph $\overline{G}= (V,E^c)$ where
$E^c$ is the complement $E(K_n) \setminus E(G)$. Given two graphs $G,H$, denote by 
$G \oplus H$ the disjoint union of $G$ and $H$. We have $(G + H)^c = G^c \oplus H^c$. 

\paragraph{}
A graph $G$ is an {\bf Zykov prime} if $G$ can not be written as $G=A+B$, where $A,B$ are 
graphs. The primes in the Zykov monoid are the graphs for which the graph 
complement $\overline{G}$ is connected. 
As there is a unique prime factorization in the dual monoid, we have a unique
additive prime factorization in the Zykov monoid. 
The map $i$ which maps a sphere to its genus satisfies $j(G + H) = j(G) j(H)$. 
We can extend Euler characteristic to negative graphs and since $j(0)=1$ we should define
$j(-G)=j(G)$. 

\paragraph{} Examples of contractible graphs are complete graphs, star graphs, 
wheel graphs or line graphs. We can build a contractible graph recursively 
by choosing a contractible subgraph and building a cone extension over this. 
In general, if we make a cone extension over $A$, then the Euler characteristic
changes by $1-\chi(A))$. This change is called the Poincar\'e-Hopf index. 
Examples of 1-spheres are cyclic graphs. From all the classical platonic solids, 
the octahedron and the icosahedron are 2-spheres. From the classical 3-polytopes, 
as shown, only the 16 cell and the 600 cell are e-spheres. 

\begin{figure}[!htpb]
\scalebox{1.0}{\includegraphics{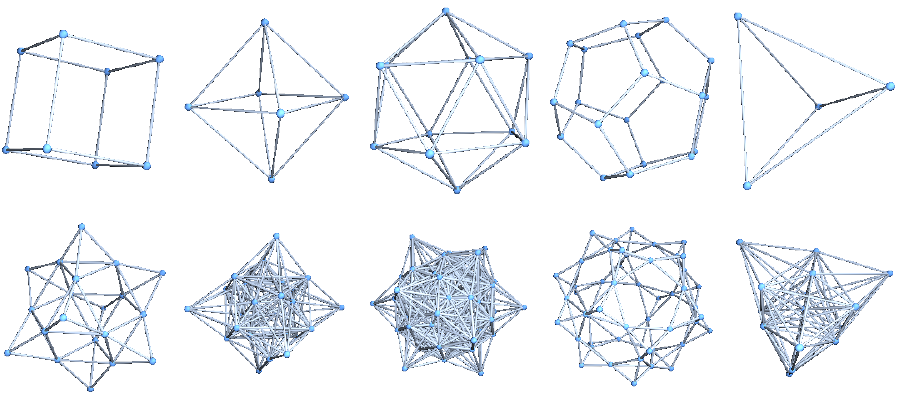}}
\label{Figure 10}
\caption{
The five platonic solid complexes and their connection graphs.
Only the octahedron and icosahedron are $d=2$-spheres with $\chi(G)=2$.
The cube has Euler characteristic $-4$, the dodecahedron has 
Euler characteristic $-10$, the tetrahedron has Euler characteristic $-1$.
All connection graphs are homotopic to their platonic solid graphs except
for the octahedron, where the connection graph is a homology 3-sphere and 
has maximal dimension $8$ (there are 9 complete subgraphs intersecting all each other
at each vertex so that there is $K_9$ subgraph of the connection graph. 
}
\end{figure}

\paragraph{}
Results on graphs immediately go over to {\bf finite abstract simplicial complexes} $G$ because 
the Barycentric graph $\phi(G)$ of $G$ is a graph.
A simplicial complex is a $d$-sphere if its Barycentric refinement is a $d$-sphere. 
The class of objects can be generalized discrete CW-complexes,
where cells take the role of balls but the boundary of a ball
does not have to be a skeleton of a simplex. Every cell has a dimension attached.
We can see the cube or the dodecahedron as a sphere. 
The definition of Platonic sphere must be adapted in that we require every 
unit sphere $S(x)$ is isomorphic to a fixed Platonic $d-1$ sphere which only 
depends on the dimension of $x$. Now the classification is the Schl\"afli 
classification $(1,1,\infty,5,6,3,3,3,3,\dots)$. In the graph case, the numbers are
$(1,1,\infty,2,2,1,1,1,1,\dots)$.

\begin{figure}[!htpb]
\scalebox{1.0}{\includegraphics{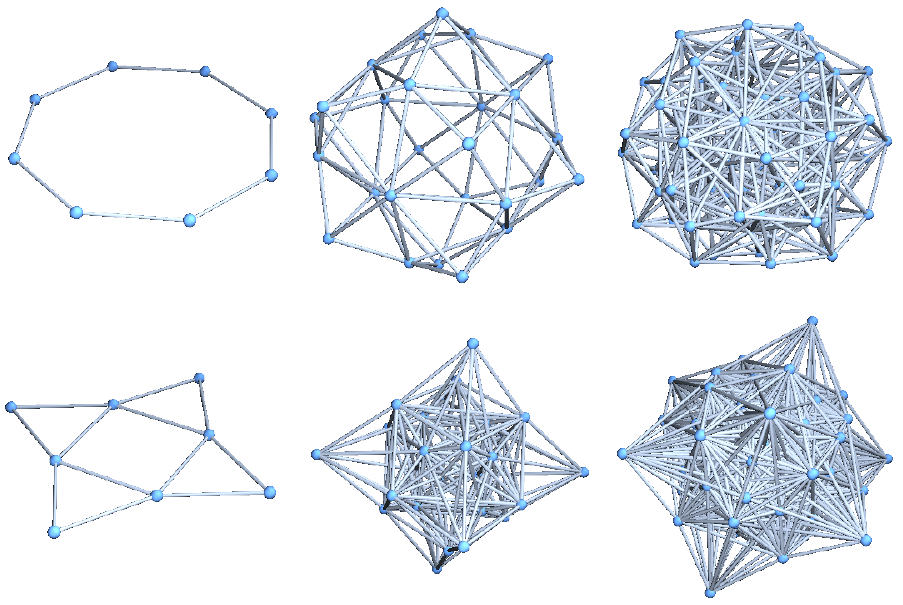}}
\label{Figure 15}
\caption{
The Barycentric graphs $\phi(S^k)$ of the one,two and three dimensional spheres. 
The 3-sphere is the 16-cell. 
Below the connection graphs $\psi(S^k)$ of the 1,2,3 spheres. They are homotopic. 
}
\end{figure}

\paragraph{}
\cite{Weyl1925} asked to characterize spheres combinatorially.
The question whether some combinatorial data can answer this without homotopy is still unknown. 
Robin Forman \cite{forman95} defined spheres using Morse theory as classically through Reeb:
spheres can be characterized as manifolds which admit a Morse function
with exactly 2 critical points. 
The inductive definition of \cite{Ivashchenko1993}, simplified by \cite{CYY}
goes back to Whitehead and is equivalent \cite{knillreeb}.
The story of polyhedra is told in \cite{Richeson,coxeter}. Historically, it was
\cite{Schlafli,Schoute,AliciaBooleStott}.

\paragraph{}
In \cite{knillgraphcoloring2}, Platonic spheres were defined $d$-spheres for which all
unit spheres are Platonic $(d-1)$-spheres. Gauss-bonnet \cite{cherngaussbonnet}
have the classification all $1$-dimensional spheres $C_n, n>3$ are
Platonic for $d=1$, the Octahedron and Icosahedron are the two Platonic $2$-spheres,
the sixteen and six-hundred cells are the Platonic $3$-spheres. 
For earlier appearances of Gauss-Bonnet theorem \cite{cherngaussbonnet}
see \cite{I94a,levitt1992,forman2000}.

\paragraph{}
The Euler gem episode illustrates the value of precise definitions are \cite{Richeson}
especially if the continuum is involved. Already when working with $1$-spheres, 
which are often modeled as polygons, one can get into muddy waters capturing what a
polygon embedded into Euclidean space is Are self-intersections allowed? Do
polygons have to be convex? The one-dimensional case gives a taste about the 
confusions which triggered the ``proof and refutation" dialog  \cite{Lakatos}.
Being free from Euclidean realizations places the topic into combinatorics
so that the results can be accepted by a finitist like Brouwer who would not even 
accept the existence of the real number line.

\bibliographystyle{plain}

\end{document}